\documentclass[a4paper,12pt,reqno]{amsart}
\usepackage{amsmath,amsfonts,amssymb,amsthm,enumerate}
\usepackage{tikz,graphicx}
\usepackage{subfigure, hyperref}
\usepackage{cleveref}
\usepackage{caption}
\usepackage{enumitem}
\usepackage{float,wrapfig}
\usepackage[labelsep=space]{caption}
\usepackage[font=footnotesize]{caption}
\newtheorem{lemma}{Lemma}
\newtheorem{theorem}{Theorem}[section]

\newtheorem{corollary}{Corollary}[section]
\newtheorem{remark}{Remark}[section]

\numberwithin{equation}{section}
\title[]{Stability of the TODA EQUATIONS RELATED TO A PERTURBED $R_I$ TYPE RECURRENCE RELATION }
\author{Vinay Shukla$^\dagger$}
\address{$^\dagger$Department of Mathematics, SCSET\\ Bennett University, Greater Noida-201310, Uttar Pradesh, India}
\email{vinayshukla4321@gmail.com}
\author{A. Swaminathan\, $^{\#\ddagger}$}{\thanks{$^{\#}$Corresponding author}}
   \address{$^\ddagger$Department of Mathematics\\ Indian Institute of Technology, Roorkee-247667, Uttarakhand, India}
   \email{mathswami@gmail.com, a.swaminathan@ma.iitr.ac.in}
\allowdisplaybreaks
\bigskip
\allowdisplaybreaks

\begin{document}
\leftline{ \scriptsize \it  }

\begin{abstract}
In this manuscript, a modified $R_I$ type recurrence relation 
is considered whose recurrence coefficients are perturbed by addition or multiplication of a constant. The perturbed system of recurrence coefficients is represented by Toda lattice equations, which are derived. These equations are then represented in a matrix form. With the help of this matrix representation, a known Lax pair is recovered. Inferences about the stability of resulting perturbed system of Toda equations are drawn based on numerical experiments.
\end{abstract}

\maketitle

\markboth{Vinay Shukla and A. Swaminathan}{Perturbed Toda}

\section{Introduction}
Orthogonal polynomials on the real line (OPRL) satisfy the following three term recurrence relation (TTRR)
\begin{align}\label{OPRL TTRR R1 paper}
	\mathcal{R}_{n+1}(x) = (x-\beta_{n+1})\mathcal{R}_n(x)-\gamma_{n+1} \mathcal{R}_{n-1}(x), \quad \mathcal{R}_{-1}(x) = 0, \quad \mathcal{R}_0(x) = 1, \quad n \geq 0.
\end{align}
The two parameters $\beta_n$ and $\gamma_n$, involved in \eqref{OPRL TTRR R1 paper}, can be modified by two fundamental mathematical operations, i.e., either by addition or by multiplication of a real number. The process is called co-recursion when the first term of the sequence $\{{\beta}_{n}\}_{n \geq 1}$ is perturbed by adding a constant, and co-dilation refers to the modification of the first term of the sequence $\{{\gamma}_n\}_{n \geq 1}$ by multiplying it with a constant \cite{chihara PAMS 1957,Dini thesis 1988}. The perturbation in $\{{\beta}_{n}\}_{n \geq 1}$ at $n=k$ is termed generalized co-recursion, while the perturbation in $\{{\gamma}_n\}_{n \geq 1}$ at $n=k$ is termed generalized co-dilation. When both $\{{\beta}_{n}\}_{n \geq 1}$ and $\{{\gamma}_n\}_{n \geq 1}$ are perturbed at $n=k$, this condition is referred to as generalized co-modification \cite{Paco perturbed recurrence 1990}. In this context, co-modified OPRL or co-polynomials on the real line (COPRL), are introduced and studied in \cite{Castillo co-polynomials on real line 2015}.

With regard to the measure $\psi(t)$ (a positive measure on the real line), let $\mathcal{P}_n(x;t)$ represent the monic orthogonal polynomials, where $t$ is the time variable that satisfies the following classical TTRR
\begin{align*}
&	\mathcal{R}_{n+1}(x;t)=(x-\beta_{n+1}(t))\mathcal{R}_n(x;t)-\gamma_{n+1}(t)\mathcal{R}_{n-1}(x;t), \quad n \geq 0, \\
	& \nonumber	\mathcal{R}_{-1}(x;t) = 0, \qquad \mathcal{R}_0(x;t) = 1,
\end{align*}
and $d\psi(t)=e^{-tx}d\psi$. It is known \cite{Esmail book} that the recurrence coefficients $\gamma_{n}(t)$ and $b_{n+1}(t)$ satisfy the semi-infinite Toda lattice equations of motion
\begin{align}\label{Toda TTRR}
	\dot{\gamma}_{n}(t) &= \gamma_{n}(t)[\beta_{n-1}(t)-\beta_{n}(t)],\nonumber \\
	\dot{\beta}_{n}(t) &= \gamma_{n}(t)-\gamma_{n+1}(t),  \quad n \geq 1,
\end{align}
with initial conditions $\beta_0(t)=1$, $\gamma_1(t)=0$, $\gamma_n(0)=\gamma_n$ and $\beta_n(0)=\beta_n$, where $\dot{f}$ represents the differentiation of $f$ with respect to $t$ in this case.

Consider the recurrence
\begin{align}\label{R1}
&  \mathcal{P}_{n+1}(z) = (z-c_{n+1})\mathcal{P}_n(z)-\lambda_{n+1} (z-a_{n+1})\mathcal{P}_{n-1}(z), \quad n\geq 0, \\ 
& \nonumber	\mathcal{P}_{-1}(z) = 0, \qquad \mathcal{P}_0(z) = 1, 
\end{align}
where $\{{\lambda}_n\}_{n \geq 1}$ and $\{c_n\}_{n \geq 1}$ are real sequences. A rational function $\psi_n(z)=\frac{\mathcal{P}_n(z)}{\prod_{j=1}^{n}(z-a_j)}$ exists if $\lambda_n \neq 0$ and $\mathcal{P}_n(a_n)\neq 0$ for $n\geq 1$, as was proven in \cite[Theorem 2.1]{Esmail masson JAT 1995}. Further, we have a linear moment functional $\mathcal{L}$ such that the orthogonality relations
\begin{align*}
\mathcal{L}[a_0]\neq 0, \quad \mathcal{L}\left[z^k \psi_n(z) \right]\neq  0, \quad 0\leq k < n,
\end{align*}
hold. According to \cite{Esmail masson JAT 1995}, \eqref{R1} is called the recurrence relation of $R_I$ type, and $\mathcal{P}_n(z), n\geq 1$, generated by it are called $R_I$ polynomials. Furthermore, \eqref{R1} is associated with a non-trivial positive measure of orthogonality defined either on the unit circle or on a subset of the real line whenever $a_n=0$, $n\geq 1$ \cite{Simon part1 2005}. Fixing $a_n=0$ and $z \in \mathbb{R}$, we arrive at a special form of $R_I$ type recurrence relation
\begin{align}\label{special R1 intro}
&	\mathcal{P}_{n+1}(x) = (x-c_{n+1})\mathcal{P}_n(x)-\lambda_{n+1}x \mathcal{P}_{n-1}(x), \quad n \geq 1,\\
	& \nonumber	\mathcal{P}_{-1}(x) = 0, \qquad \mathcal{P}_0(x) = 1
\end{align}
satisfied by orthogonal Laurent polynomials, where $\{c_n\}_{n \geq 1}$ is a real sequence and $\{{\lambda}_n\}_{n \geq 1}$ are positive real numbers. We refer to \cite{Castillo costa ranga veronese Favard theorem 2014,Costa felix ranga 2013,Kim stanton 2021} for some recent progress related to $R_I$ type recurrence relations and $R_I$ polynomials. The modification of recurrence coefficients in \eqref{R1} and \eqref{special R1 intro} at $n=k$, defined as
\begin{align}
	c_{k+1} &\rightarrow c_{k+1}+\mu_{k+1} , \qquad \rm(generalized~ co-recursion) \label{co-recursive condition DL}\\
	\lambda_{k+1} &\rightarrow \nu_{k+1}\lambda_{k+1}, \hspace{0.5cm}\qquad \rm(generalized~ co-dilation), \label{co-dilated condition DL}
\end{align}
where $k$ is a fixed non-negative integer, has been studied in \cite{swami vinay BMMS 2023} and resultant polynomials are termed as co-polynomials of $R_I$ type.

Let $\mathfrak{N}$ denote the moment functional such that the moments given by $\mathfrak{N}[x^k e^{-t(px+\frac{q}{x})}]$ exists where $p,q \in \mathbb{C}$ and $ k \in \mathbb{Z}$ and . The parametrized moment functional  $\mathfrak{N}^{(t)}$ for $t \geq 0$ is such that
\begin{align}\label{moment functional L_t}
	\mathfrak{N}^{(t)}[x^k]=\mathfrak{N}[x^k e^{-t(px+\frac{q}{x})}].	
\end{align}
Following \cite{Vinet Zhedavon integrable chain 1998}, if $\mathcal{P}_n(x;t)$ satisfy
\begin{align}\label{R_n Toda}
	\mathcal{P}_{n+1}(x;t) = (x-c_{n+1}(t))\mathcal{P}_n(x;t)-\lambda_{n+1}(t)x \mathcal{P}_{n-1}(x;t), \quad n \geq 1,
\end{align}
with $\mathcal{P}_0(x;t)=1$ and $\mathcal{P}_1(x;t)=x-c_1(t)$. Then, the recurrence coefficients $\lambda_n(t)$ and $c_{n}(t)$ satisfy relativistic Toda equations
\begin{equation}\label{Relativistic toda 1}
	 \dot{\lambda}_{n}=\lambda_n \left(\dfrac{1}{c_{n-1}}-\dfrac{1}{c_{n}}\right) ,
	\quad \dot{c}_{n} =c_{n}\left(\dfrac{\lambda_{n+1}}{c_{n+1}c_{n}}-\dfrac{\lambda_{n}}{c_{n}c_{n-1}}\right), \quad n \geq 1,
\end{equation}
when $q=1$ and $p=0$, and
\begin{align}\label{Relativistic toda 2}
	\dot{\lambda}_{n}=\lambda_{n}(\lambda_{n-1}+c_{n-1}-\lambda_{n+1}-c_{n}), \quad
	\dot{c}_{n} = c_{n}(\lambda_{n}-\lambda_{n+1}), \quad  n \geq 1,
\end{align}
when $q=0$ and $p=1$. Moreover, using a more general form of \eqref{R_n Toda} given as
\begin{align*}
	\mathcal{P}_{n+1}(x;t) = (x-c_{n+1}(t))\mathcal{P}_n(x;t)-\lambda_{n+1}(t)(x-a_n(t)) \mathcal{P}_{n-1}(x;t), \quad n \geq 1,
\end{align*}
with $\mathcal{P}_0(x;t)=1$ and $\mathcal{P}_1(x;t)=x-c_1(t)$, the relativistic Toda equations \eqref{Relativistic toda 1} and \eqref{Relativistic toda 2} have been further extended to generalized relativistic Toda equations (see \cite{Vinet Zhedavon integrable chain 1998}) that are expressed as
\begin{align*}
	&\dot{c}_{n}=a_{n}\lambda_{n}-a_{n+1}\lambda_{n+1}+c_{n}(\lambda_{n+1}-\lambda_{n}),\\
	&\dot{\lambda}_{n}=\lambda_{n}(\lambda_{n+1}-\lambda_{n-1}+c_{n}-c_{n-1}), \quad n \geq 1.
\end{align*}
In a recent work \cite{Bracciali ranga toda 2019}, extended relativistic Toda equations have been derived corresponding to the L-orthogonal polynomials \eqref{R_n Toda} via duo-directional modification $px +q/x$ in \eqref{moment functional L_t}. We refer to \cite{Fern Manas 2024, Deano Morey 2024} and references therein for some recent advancements in the theory of Toda equations. This manuscript aims to investigate the consequences of perturbing the recurrence coefficients in \eqref{R_n Toda} on the Toda equations of motion. Perturbed extended relativistic Toda equations are the name given to the equations so derived.
The Toda lattice equation for $n=k$ is found by considering a single perturbation of the recurrence coefficient; this is the focus of \Cref{Toda lattice equations}. A Lax pair for extended relativistic Toda equations, discussed in \cite{Bracciali ranga toda 2019}, is also recovered. A natural question arises regarding the stability of such a perturbed system. This aspect has been  addressed in this work, with \Cref{Stability} dedicated to investigating how the system's dynamics are influenced and whether the perturbed system demonstrates stable behavior. To the author's knowledge, such a study has not been conducted in the existing literature for Toda equations.

\section{Toda lattice equations}\label{Toda lattice equations}
\subsection{Perturbed extended relativistic Toda equations}
The monic polynomials $\mathcal{\hat{P}}_n(x;t)$ of degree $n$ satisfy
\begin{align}\label{R_n Toda perturbed}
\mathcal{\hat{P}}_{n+1}(x;t) = (x-\hat{c}_{n+1}(t))\mathcal{\hat{P}}_n(x;t)-\hat{\lambda}_{n+1}(t)x \mathcal{\hat{P}}_{n-1}(x;t), \quad n \geq 1,
\end{align}
where $\hat{c}_{n+1}(t)$ and $\hat{\lambda}_{n+1}(t)$ denote the modification of recurrence coefficients as described in \eqref{co-recursive condition DL} and \eqref{co-dilated condition DL}, e.g.,
\begin{align*}
&\hat{c}_{k+1}(t)={c}_{k+1}(t)-\mu_{k+1}(t), \quad \hat{\lambda}_{k+1}(t)=\nu_{k+1}(t){\lambda}_{k+1}(t), \quad {\rm for}~~ n=k, \\
& \hat{c}_{n+1}(t)={c}_{n+1}(t), \hspace{2.2cm} \hat{\lambda}_{n+1}(t)={\lambda}_{n+1}(t),\hspace{1.6cm} {\rm for} ~~ n \neq k.
\end{align*}
\begin{lemma}\label{Lemma toda}
If $\hat{\rho}_n=\mathfrak{\hat{N}}^{(t)}[x\mathcal{\hat{P}}_n(x;t)]$ for $n \geq 0$, then
\begin{align*}
	\hat{\rho}_n=\hat{\sigma}_{n,n}(t)\sum_{i=1}^{n+1}[\hat{c}_{i}(t)+\hat{\lambda}_{i+1}(t)], \quad n \geq 0.
\end{align*}
\end{lemma}
The proof of the above lemma is similar to the one given in \cite{Bracciali ranga toda 2019} and hence omitted. The next result provides a perturbed analogue of Theorem $1$ given in \cite{Bracciali ranga toda 2019}.

\begin{theorem}\label{Perturbed theorem Toda}
Let $\mathcal{P}_n(x;t)$ be the L-orthogonal polynomials generated by \eqref{R_n Toda} and $\mathcal{\hat{P}}_n(x;t)$ be its perturbed version defined by \eqref{R_n Toda perturbed}. The recurrence coefficients ${\lambda}_{n}(t)$ and $c_n(t)$ satisfy the following perturbed extended relativistic Toda equations:
\begin{align}\label{PERT c_n}
&\dfrac{\dot{c}_{k}}{c_{k}}=p({\lambda}_{k}-\nu_{k+1}{\lambda}_{k+1})+q\left(\dfrac{\nu_{k+1}{\lambda}_{k+1}}{(c_{k+1}-\mu_{k+1})c_{k}}-\dfrac{{\lambda}_{k}}{c_{k}c_{k-1}}\right), \quad   n=k, \nonumber\\
&\dot{c}_{k+1}=\dot{\mu}_{k+1}+p(c_{k+1}-\mu_{k+1})(\nu_{k+1}{\lambda}_{k+1}-{\lambda}_{k+2})+q\left(\dfrac{{\lambda}_{k+2}}{c_{k+2}}-\dfrac{\nu_{k+1}{\lambda}_{k+1}}{c_{k}} \right), \quad   n=k+1,\nonumber \\
& \dfrac{\dot{c}_{k+2}}{c_{k+2}}=p({\lambda}_{k+2}-{\lambda}_{k+3})+q\left(\dfrac{{\lambda}_{k+3}}{c_{k+2}c_{k+3}}-\dfrac{{\lambda}_{k+2}}{(c_{k+1}-\mu_{k+1})c_{k+2}} \right), \quad  n=k+2, \nonumber\\
&
\dfrac{\dot{c}_{n}}{c_{n}}=p({\lambda}_{n}-{\lambda}_{n+1})+q\left(\dfrac{{\lambda}_{n+1}}{c_{n+1}c_{n}}-\dfrac{{\lambda}_{n}}{c_{n}c_{n-1}} \right), \quad otherwise,
\end{align}
and for $n=k$,
\begin{align}\label{PERT lambda_n}
& \dfrac{\dot{\lambda}_{k+1}}{{\lambda}_{k+1}}=-\dfrac{\dot{\nu}_{k+1}}{{\nu}_{k+1}}- p({\lambda}_{k+2}+c_{k+1}-\mu_{k+1}+(\nu_{k+1}-1)\lambda_{k+1}-c_k-\lambda_{k})-q\left(\dfrac{1}{c_{k+1}-\mu_{k+1}}-\dfrac{1}{c_k} \right),\nonumber\\
& \dfrac{\dot{\lambda}_{k+2}}{{\lambda}_{k+2}}=-p[\lambda_{k+3}+c_{k+2}-(c_{k+1}-\mu_{k+1})-\nu_{k+1}\lambda_{k+1}]-q\left(\dfrac{1}{c_{k+2}}-\dfrac{1}{c_{k+1}-\mu_{k+1}} \right), \quad  n=k+1, \nonumber\\
&\dfrac{\dot{\lambda}_{n}}{{\lambda}_{n}}=p(\lambda_{n-1}+c_{n-1}-\lambda_{n+1}-c_n)+q\left(\dfrac{1}{c_{n-1}}-\dfrac{1}{c_n} \right), \quad otherwise,
\end{align}
with initial conditions $c_0(t)=1$, $\lambda_0(t)=-1$ and $\lambda_1(t)=0$. Here, variable $t$ has been omitted for simplicity.
\end{theorem}
\begin{proof}
The sequences $\mathcal{P}_n(x;t)$ and $\mathcal{\hat{P}}_n(x;t)$ satisfy the following L-orthogonality conditions
\begin{align}
\mathfrak{\hat{N}}^{(t)}[x^{-n+j}\mathcal{\hat{P}}_n(x;t)]=0, \quad j= 0,1, \ldots n-1, \label{Hat L toda}\\
\mathfrak{N}^{(t)}[x^{-n+r}\mathcal{P}_n(x;t)]= 0, \quad r= 0,1, \ldots n-1. \nonumber
\end{align}
Let $\mathcal{P}_n(x;t)=\sum_{j=0}^{n}a_{n,j}x^j$ and $\mathcal{\hat{P}}_n(x;t)=\sum_{j=0}^{n}\hat{a}_{n,j}x^j$ where $a_{n,n}=1$ and $\hat{a}_{n,n}=1$. Then $\sigma_{n,-1}(t)=\mathfrak{N}^{(t)}[x^{-n-1}\mathcal{P}_n(x;t)]$, $\hat{\sigma}_{n,n}(t)=\mathfrak{N}^{(t)}[\mathcal{P}_n(x;t)]$, $\hat{\sigma}_{n,-1}(t)=\mathfrak{\hat{N}}^{(t)}[x^{-n-1}\mathcal{\hat{P}}_n(x;t)]$ and $\hat{\sigma}_{n,n}(t)=\mathfrak{\hat{N}}^{(t)}[\mathcal{\hat{P}}_n(x;t)]$. Clearly, for $n < k$,  $\hat{\sigma}_{n,n}(t)=\sigma_{n,n}(t)$ and $\hat{\sigma}_{n,-1}(t)=\sigma_{n,-1}(t)$.
From this, we have
\begin{align}\label{sigma_k,-1}
\begin{split}
&\hat{\sigma}_{k,-1}(t)=-\dfrac{\nu_{k+1}(t)\lambda_{k+1}(t)}{{c}_{k+1}(t)-\mu_{k+1}(t)}\sigma_{k-1,-1}(t), \\
&\hat{\sigma}_{n,-1}(t)=-\dfrac{\lambda_{n+1}(t)}{{c}_{n+1}(t)}\hat{\sigma}_{n-1,-1}(t), \quad n >k,
\end{split}
\end{align}
and
\begin{align}\label{sigma_k,k}
\begin{split}
&\hat{\sigma}_{k,k}(t)=\nu_{k+1}(t)\lambda_{k+1}(t)\sigma_{k-1,k-1}(t), \\
& \hat{\sigma}_{n,n}(t)=\lambda_{n+1}(t)\hat{\sigma}_{n-1,n-1}(t), \quad n >k.
\end{split}
\end{align}
From \eqref{R_n Toda} and \eqref{R_n Toda perturbed}, we have
\begin{align*}
a_{n,0}(t)= (-1)^nc_n(t)c_{n-1}(t) \ldots c_2(t) c_1(t), \quad n \geq 1, \\
\hat{a}_{n,0}(t)= (-1)^n\hat{c}_n(t)\hat{c}_{n-1}(t) \ldots \hat{c}_2(t) \hat{c}_1(t), \quad n \geq 1,	
\end{align*}
which implies
\begin{align}\label{a_n,0}
\begin{split}
&\hat{a}_{n,0}(t)=a_{n,0}(t), \quad n\leq k, \quad \hat{a}_{k+1,0}(t)=-({c}_{k+1}(t)-\mu_{k+1}(t))a_{k,0}(t),\\
\text{and} \quad & \hat{a}_{n,0}(t)=-c_n(t)\hat{a}_{n-1,0}(t), \quad n > k+1.
\end{split}
\end{align}
Again, using \eqref{R_n Toda perturbed}, we obtain
\begin{align*}
\hat{c}_{n+1}(t)+\hat{\lambda}_{n+1}(t)=\hat{a}_{n,n-1}(t)-\hat{a}_{n+1,n}(t), \quad n \geq 1,
\end{align*}
which implies
\begin{align}\label{a_n+1,n-a_n,n-1}
\begin{split}
&c_{n+1}(t)+\lambda_{n+1}(t)={a}_{n,n-1}(t)-{a}_{n+1,n}(t), \quad n < k, \\
&c_{k+1}(t)+\lambda_{k+1}(t)-\mu_{k+1}(t)+(\nu_{k+1}(t)-1)\lambda_{k+1}(t)=	{a}_{k,k-1}(t)-\hat{a}_{k+1,k}(t),\\
&c_{n+1}(t)+\lambda_{n+1}(t)= \hat{a}_{n,n-1}(t)-\hat{a}_{n+1,n}(t), \quad n > k.
\end{split}
\end{align}
From \eqref{Hat L toda}, we get
\begin{align*}
\mathfrak{\hat{N}}^{(t)}[x^{-n}\mathcal{\hat{P}}_{n-1}(x;t)\mathcal{\hat{P}}_n(x;t)]=0, \quad n \geq 1.
\end{align*}
Differentiating this with respect to (w.r.t.) $t$, we have
\begin{align}\label{Differentiation wrt T toda}
\mathfrak{\hat{N}}^{(t)}&[x^{-n}\mathcal{\dot{\hat{P}}}_{n-1}(x;t)\mathcal{\hat{P}}_n(x;t)]+
\mathfrak{\hat{N}}^{(t)}[x^{-n}\mathcal{\hat{P}}_{n-1}(x;t)\mathcal{\dot{\hat{P}}}_n(x;t)] \nonumber \\
&-\mathfrak{\hat{N}}^{(t)}[x^{-n}\left(px+\frac{q}{x}\right)\mathcal{\hat{P}}_{n-1}(x;t)\mathcal{\hat{P}}_n(x;t)]=0, \quad n \geq 1.
\end{align}
where $\mathcal{\dot{\hat{P}}}_n(x;t)=\dfrac{\partial}{\partial t}\mathcal{\hat{P}}_n(x;t)$. Again from \eqref{Hat L toda}, we get $\mathfrak{\hat{N}}^{(t)}[x^{-n}\mathcal{\dot{\hat{P}}}_{n-1}(x;t)\mathcal{\hat{P}}_n(x;t)] = 0$. Further, since $\mathcal{\dot{\hat{P}}}_n(x;t)=\sum_{j=0}^{n-1}\dot{\hat{a}}_{n,j}x^j$ , we obtain
\begin{align*}
&\mathfrak{\hat{N}}^{(t)}[x^{-n}\mathcal{\hat{P}}_{n-1}(x;t)\mathcal{\dot{\hat{P}}}_n(x;t)]=\dot{\hat{a}}_{n,0}(t)\hat{\sigma}_{n-1,-1}(t), \\
&\mathfrak{\hat{N}}^{(t)}[x^{-n+1}\mathcal{\hat{P}}_{n-1}(x;t)\mathcal{\hat{P}}_n(x;t)]=\hat{\sigma}_{n,n}(t), \\
&\mathfrak{\hat{N}}^{(t)}[x^{-n-1}\mathcal{\hat{P}}_{n-1}(x;t)\mathcal{\hat{P}}_n(x;t)]=\hat{\sigma}_{n,n}(t)=\hat{a}_{n-1,0}(t)\hat{\sigma}_{n,-1}(t), \quad n \geq 1.
\end{align*}
Putting these values in \eqref{Differentiation wrt T toda}, we have
\begin{align*}
\dot{\hat{a}}_{n,0}(t)\hat{\sigma}_{n-1,-1}(t)=p\hat{\sigma}_{n,n}(t)+q\hat{a}_{n-1,0}(t)\hat{\sigma}_{n,-1}(t), \quad n \geq 1.
\end{align*}
Substituting \eqref{sigma_k,-1}, \eqref{sigma_k,k} and \eqref{a_n,0} in the above expression, we conclude
\begin{align}
\sum_{i=1}^{n}\dfrac{\dot{c}_{i}(t)}{c_{i}(t)}&=-p\lambda_{n+1}(t)+q\dfrac{\lambda_{n+1}(t)}{c_{n+1}(t)c_{n}(t)}, \quad n < k, \label{Toda c_k 1} \\
\sum_{i=1}^{k}\dfrac{\dot{c}_{i}(t)}{c_{i}(t)}&=-p\nu_{k+1}(t)\lambda_{k+1}(t)+q\dfrac{\nu_{k+1}(t)\lambda_{k+1}(t)}{(c_{k+1}(t)-\mu_{k+1}(t))c_{k}(t)}, \quad n = k, \label{Toda c_k 2} \\
\sum_{i=1}^{k+1}\dfrac{\dot{\hat{c}}_{i}(t)}{\hat{c}_{i}(t)}&=\sum_{i=1}^{k}\dfrac{\dot{c}_{i}(t)}{c_{i}(t)}+\dfrac{\dot{c}_{k+1}(t)-\dot{\mu}_{k+1}(t)}{c_{k+1}(t)-\mu_{k+1}(t)}=-p\lambda_{k+2}(t)\nonumber \\
&+q\dfrac{\lambda_{k+2}(t)}{{(c_{k+1}(t)-\mu_{k+1}(t))c_{k+2}(t)}}, \quad n = k+1, \label{Toda c_k 3}\\
\sum_{i=1}^{n}\dfrac{\dot{\hat{c}}_{i}(t)}{\hat{c}_{i}(t)}&=-p\lambda_{n+1}(t)+q\dfrac{\lambda_{n+1}(t)}{c_{n+1}(t)c_{n}(t)}, \quad n > k+1. \label{Toda c_k 4}
\end{align}
Substraction of \eqref{Toda c_k 3} from \eqref{Toda c_k 4} for $n=k+2$, \eqref{Toda c_k 2} from \eqref{Toda c_k 3}, and \eqref{Toda c_k 1} from \eqref{Toda c_k 2} for $n=k-1$ yields the first four relations given in the hypothesis of the \Cref{Perturbed theorem Toda}.\par
Note that \eqref{Hat L toda} implies $\mathfrak{\hat{N}}^{(t)}[x^{-n}\mathcal{\hat{P}}^2_n(x;t)] =\hat{\sigma}_{n,n}(t)$. Differentiating this w.r.t. $t$ and substituting $\mathfrak{\hat{N}}^{(t)}[x^{-n}\mathcal{\dot{\hat{P}}}_{n}(x;t)\mathcal{\hat{P}}_n(x;t)] = 0$, we obtain
\begin{align}\label{Above Lemma toda}
-\mathfrak{\hat{N}}^{(t)}[x^{-n}\left(px+\frac{q}{x}\right)\mathcal{\hat{P}}^2_n(x;t)]=
\dot{\hat{\sigma}}_{n,n}(t), \quad n \geq 1.
\end{align}
Now, from \eqref{Hat L toda}, we observe that
\begin{align}\label{Below lemma eq 1}
\mathfrak{\hat{N}}^(t) [x^{-n-1}\mathcal{\hat{P}}^2_n(x;t)] = \hat{a}_{n,0}(t)\hat{\sigma}_{n,-1}(t), \quad n \geq 0	,
\end{align}
and using \Cref{Lemma toda}, we get
\begin{align}\label{Below lemma eq 2}
\mathfrak{\hat{N}}^(t) [x^{-n+1}\mathcal{\hat{P}}^2_n(x;t)] =  \hat{a}_{n,n-1}(t)\hat{\sigma}_{n,n}(t)+\hat{\sigma}_{n,n}(t)\sum_{i=1}^{n+1}[\hat{c}_{i}(t)+\hat{\lambda}_{i+1}(t)],\quad n \geq 0.
\end{align}
Using \eqref{a_n+1,n-a_n,n-1}, we obtain
\begin{align}\label{Below lemma eq 3}
\begin{split}
&{a}_{n,n-1}(t)+\sum_{i=1}^{n+1}[{c}_{i}(t)+{\lambda}_{i+1}(t)]={\lambda}_{n+2}(t)+{\lambda}_{n+1}(t)+{c}_{n+1}(t),\quad n < k,\\
&{a}_{k,k-1}(t)+\sum_{i=1}^{n+1}[\hat{c}_{i}(t)+\hat{\lambda}_{i+1}(t)]={\lambda}_{k+2}(t)+\nu_{k+1}(t){\lambda}_{k+1}(t)+{c}_{k+1}(t)-\mu_{k+1}(t),\quad n = k,\\
&\hat{a}_{n,n-1}(t)+\sum_{i=1}^{n+1}[\hat{c}_{i}(t)+\hat{\lambda}_{i+1}(t)]={\lambda}_{n+2}(t)+{\lambda}_{n+1}(t)+{c}_{n+1}(t),\quad n > k.
\end{split}
\end{align}
Substituting \eqref{Below lemma eq 1}, \eqref{Below lemma eq 2} and \eqref{Below lemma eq 3} in \eqref{Above Lemma toda} and dividing by $\hat{\sigma}_{n,n}(t)$, we conclude that
\begin{align}
-p[\lambda_{n+2}(t)&+\lambda_{n+1}(t)+c_{n+1}(t)]-\dfrac{q}{c_{n+1}(t)}=\dfrac{\dot{\sigma}_{0,0}(t)}{\sigma_{0,0}(t)}+\sum_{i=2}^{n+1}\dfrac{\dot{\lambda}_{i}(t)}{\lambda_{i}(t)}, \quad n < k, \label{Toda lambda_k 1}\\
-p[\lambda_{k+2}(t)&+\nu_{k+1}\lambda_{k+1}(t)+c_{k+1}(t)-\mu_{k+1}(t)]-\dfrac{q}{c_{k+1}(t)-\mu_{k+1}(t)} \nonumber \\
&=\dfrac{\dot{\sigma}_{0,0}(t)}{\sigma_{0,0}(t)}+\sum_{i=2}^{n+1}\dfrac{\dot{\lambda}_{i}(t)}{\lambda_{i}(t)}+\dfrac{\dot{\nu}_{k+1}(t)}{\nu_{k+1}(t)}=\dfrac{\dot{\sigma}_{0,0}(t)}{\sigma_{0,0}(t)}+\sum_{i=2}^{n+1}\dfrac{\dot{\hat{\lambda}}_{i}(t)}{\hat{\lambda}_{i}(t)}, \quad n = k, \label{Toda lambda_k 2}\\
-p[\lambda_{n+2}(t)&+\lambda_{n+1}(t)+c_{n+1}(t)]-\dfrac{q}{c_{n+1}(t)}=\dfrac{\dot{\sigma}_{0,0}(t)}{\sigma_{0,0}(t)}+\sum_{i=2}^{n+1}\dfrac{\dot{\hat{\lambda}}_{i}(t)}{\hat{\lambda}_{i}(t)}, \quad n > k. \label{Toda lambda_k 3}
\end{align}
Substraction of \eqref{Toda lambda_k 3} from \eqref{Toda lambda_k 2} for $n=k+1$ and \eqref{Toda lambda_k 1} for $n=k-1$ from \eqref{Toda lambda_k 2} yields the last three relations given in the hypothesis of the \Cref{Perturbed theorem Toda}. This completes the proof. \qedhere\par

\end{proof}

\begin{corollary}
With the help of \Cref{Perturbed theorem Toda}, the perturbed version of \eqref{Relativistic toda 1} and \eqref{Relativistic toda 2}, the relativistic Toda equations, can be obtained by substituting either $p=0$, $q=1$ or $p=1$, $q=0$.
\end{corollary}

\begin{remark}
Clearly, we get \cite[Theorem 1]{Bracciali ranga toda 2019} by substituting $\mu_{k+1}=0$ and $\nu_{k+1}=1$ in \Cref{Perturbed theorem Toda}.
\end{remark}


\begin{remark}\label{Main Remark 1}
It is noteworthy that a single modification in the recurrence coefficients for $n=k$ affects three Toda equation of motion for $c_n$, i.e., for $n=k$, $n= k+1$ and $n=k+2$, and two Toda equation of motion for $\lambda_{n}$ i.e., for $n=k$ and $n=k+1$, with the remaining levels unchanged.
\end{remark}

\subsection{Recovering a Lax pair from a matrix differential representation} A matrix differential equation corresponding to perturbed extended relativistic Toda equations given in \Cref{Perturbed theorem Toda} is derived in this subsection, which is further used to recover a Lax pair presented in \cite{Bracciali ranga toda 2019}. A matrix differential equation of the form
\begin{align*}
	\mathcal{\dot{S}}=[\mathcal{S},\mathcal{T}]=\mathcal{S}\mathcal{T}-\mathcal{T}\mathcal{S},
\end{align*}
is called a Lax representation for Toda lattice equations \eqref{Toda TTRR} and the pair $\{\mathcal{S},\mathcal{T} \}$ is called a Lax pair \cite{Deift book 1999} where $\mathcal{S}$ and $\mathcal{T}$ are semi-infinite tri-diagonal and strictly lower Hessenberg matrices in case of \eqref{Toda TTRR} \cite{Nakamura 2004}, lower and upper bi-diagonal matrices in case of relativistic Toda equations \eqref{Relativistic toda 1} and \eqref{Relativistic toda 2} \cite{walter toda laurent 2002}, and upper Hessenberg and tri-diagonal matrices in case of extended relativistic Toda equations \cite{Bracciali ranga toda 2019}.
\begin{theorem}
Let $\gamma_{n}(t)=\lambda_{n+1}(t)+c_n(t)$, $c_0(t)=1$, $\lambda_{0}=-1$ and $\lambda_{1}=0$, with further restriction $\lambda_{N+1}=0$. Then, for $1< k < N$, the perturbed extended relativistic Toda equations satisfies matrix differential equation
\begin{align}\label{MDE PERT}
\mathcal{\dot{E}}_n=\mathcal{E}_n\mathcal{X}_n-\mathcal{Y}_n\mathcal{E}_n,
\end{align}
where
\begin{align*}
\mathcal{X}_n=p\mathcal{L}'_n\mathcal{F}_n+q\mathcal{L}'_n\mathcal{G}_n+ \dfrac{\mathcal{E}^{-1}_n}{2} [p\mathcal{H}_n+q\mathcal{J}_n], \quad
\mathcal{Y}_n=p\mathcal{F}_n\mathcal{L}''_n+q\mathcal{G}_n\mathcal{L}''_n-  [p\mathcal{H}_n+q\mathcal{J}_n]\dfrac{\mathcal{E}^{-1}_n}{2},	
\end{align*}
and $\mathcal{E}_n$, $\mathcal{X}_n$, $\mathcal{Y}_n$, $\mathcal{F}_n$, $\mathcal{G}_n$, $\mathcal{H}_n$, $\mathcal{J}_n$ and $\mathcal{K}_n$ are the square matrices of order $N$ given by
\begin{align*}
\mathcal{E}_n=
\begin{pmatrix}
\gamma_1 & \gamma_2 & \ldots & \gamma_{k-1} & \gamma_k & \gamma_{k+1} & \gamma_{k+2} & \ldots & \gamma_{N}\\
\lambda_2 & \gamma_2 & \ldots & \gamma_{k-1} & \gamma_k & \gamma_{k+1} & \gamma_{k+2} & \ldots & \gamma_{N}\\
0 & \lambda_3 & \ldots & \gamma_{k-1} & \gamma_k & \gamma_{k+1} & \gamma_{k+2} & \ldots & \gamma_{N}\\
\vdots & \vdots & \ddots & \vdots & \vdots & \vdots & \vdots & & \vdots\\
0 & 0 & \ldots & \lambda_{k} & \gamma_{k} & \gamma_{k+1} & \gamma_{k+2} & \ldots & \gamma_{N}\\
0 & 0 & \ldots & 0 & \lambda_{k+1} & \gamma_{k+1} & \gamma_{k+2} & \ldots & \gamma_{N}\\
0 & 0 & \ldots & 0 & 0 & \lambda_{k+2} & \gamma_{k+2} & \ldots & \gamma_{N} \\
\vdots & \vdots & \ddots & \vdots & \vdots & \ddots & \ddots & & \vdots\\
0 & 0 & \ldots & 0 & 0 & 0 & 0 & \lambda_{N} & \gamma_{N}
\end{pmatrix},
\end{align*}

\begin{align*}
\mathcal{L}_n=
\begin{pmatrix}
\gamma_1 & \gamma_2 & \ldots & \gamma_{k-1} & \gamma_k & \gamma_{k+1} & \gamma_{k+2} & \ldots & \gamma_{N}\\
\lambda_2 & \gamma_2 & \ldots & \gamma_{k-1} & \gamma_k & \gamma_{k+1} & \gamma_{k+2} & \ldots & \gamma_{N}\\
0 & \lambda_3 & \ldots & \gamma_{k-1} & \gamma_k & \gamma_{k+1} & \gamma_{k+2} & \ldots & \gamma_{N}\\
\vdots & \vdots & \ddots & \vdots & \vdots & \vdots & \vdots & & \vdots\\
0 & 0 & \ldots & \lambda_{k} & \gamma_{k} & \gamma_{k+1} & \gamma_{k+2} & \ldots & \gamma_{N}\\
0 & 0 & \ldots & 0 & \nu_{k+1}\lambda_{k+1} & \gamma_{k+1}-\mu_{k+1} & \gamma_{k+2} & \ldots & \gamma_{N}\\
0 & 0 & \ldots & 0 & 0 & \lambda_{k+2} & \gamma_{k+2} & \ldots & \gamma_{N} \\
\vdots & \vdots & \ddots & \vdots & \vdots & \ddots & \ddots & & \vdots\\
0 & 0 & \ldots & 0 & 0 & 0 & 0 & \lambda_{N} & \gamma_{N}
\end{pmatrix},
\end{align*}

\begin{align*}
\mathcal{F}_n=
\begin{pmatrix}
\lambda_1 & 0 & 0 & \ldots  & 0 & 0 & \ldots & 0\\
-\lambda_2 & \lambda_2 & 0 & \ldots  & 0 & 0 & \ldots & 0\\
0 & -\lambda_3 & \lambda_3 & \ldots  & 0 & 0 & \ldots & 0\\
\vdots & \vdots & \ddots & \ddots  & \vdots & \vdots & & \vdots\\
0 & \ldots & 0 & \ldots  & \lambda_{k} & 0 & \ldots & 0\\
0 & \ldots & 0 & \ldots  & -\lambda_{k+1} & \lambda_{k+1} & \ldots & 0 \\
\vdots & \ddots & \vdots  & \ddots & \ddots & \ddots & \ddots & \vdots\\
0 & \ldots & 0 & \ldots & 0 & 0 & -\lambda_{N} & \lambda_{N}
\end{pmatrix},
\end{align*}

\begin{align*}
\mathcal{G}_n=
\begin{pmatrix}
\frac{1}{\gamma_1-\lambda_2} & -\frac{1}{\gamma_1-\lambda_2} & 0 & \ldots  & 0 & 0 & \ldots & 0\\
0 & \frac{1}{\gamma_2-\lambda_3} & -\frac{1}{\gamma_2-\lambda_3} & \ldots  & 0 & 0 & \ldots & 0\\
0 & 0 & \frac{1}{\gamma_3-\lambda_4} & \ldots  & 0 & 0 & \ldots & 0\\
\vdots & \vdots & \ddots & \ddots  & \ddots & \vdots & & \vdots\\
0 & \ldots & 0 & \ldots  & -\frac{1}{\gamma_{k}-\lambda_{k+1}} & 0 & \ldots & 0\\
0 & \ldots & 0 & \ldots  & \frac{1}{\gamma_{k+1}-\lambda_{k+2}-\mu_{k+1}} & -\frac{1}{\gamma_{k+1}-\lambda_{k+2}-\mu_{k+1}} & \ldots & 0 \\
\vdots & \ddots & \vdots  & \ddots & \ddots & \ddots & & \vdots\\
0 & \ldots & 0 & \ldots & 0 & 0 & 0 & \frac{1}{\gamma_{N}-\lambda_{N+1}}
\end{pmatrix},
\end{align*}
and
\begin{align*}
\mathcal{H}_n=
\begin{pmatrix}
0 & 0 & \ldots & 0 & h_1 & h_2 & 0 & \ldots & 0\\
0 & 0 & \ldots & 0 & h_1 & h_2 & 0 & \ldots & 0\\
0 & 0 & \ldots & 0 & h_1 & h_2 & 0 & \ldots & 0\\
\vdots & \vdots & \ddots & \vdots & \vdots & \vdots & \vdots & & \vdots\\
0 & 0 & \ldots & 0 & h_1 & h_2 & 0 & \ldots & 0 \\
0 & 0 & \ldots & h_4 & h_3 & h_2 & 0 & \ldots & 0 \\
0 & 0 & \ldots & 0 & h_5 & h_5 & 0 & \ldots & 0 \\
\vdots & \vdots & \ddots & \vdots & \vdots & \vdots & \vdots & & \vdots\\
0 & 0 & \ldots & 0 & 0 & 0 & 0 & \ldots & 0
\end{pmatrix},
\end{align*}
where $h_1=[-\frac{\dot{\nu}_{k+1}}{p\nu_{k+1}}+(1-\nu_{k+1})\gamma_{k}+\mu_{k+1}]\lambda_{k+1}$, $h_2=\lambda_{k+1}[\nu_{k+1}(\gamma_{k+1}-\mu_{k+1})-\gamma_{k+1}]+p^{-1}\dot{\mu}_{k+1}$, $h_3=-\frac{\dot{\nu}_{k+1}\lambda_{k+1}}{p\nu_{k+1}}+(1-\nu_{k+1})\lambda_{k}\lambda_{k+1}$, $h_4=-(1-\nu_{k+1})\lambda_{k}\lambda_{k+1}$ and $h_5=-(1-\nu_{k+1})\lambda_{k+1}\lambda_{k+2}$.

\begin{align*}
\mathcal{J}_n=
\begin{pmatrix}
0 & 0 & \ldots & 0 & j_1 & j_2 & j_3 & 0 & \ldots & 0\\
0 & 0 & \ldots & 0 & j_1 & j_2 & j_3 & 0 & \ldots & 0\\
0 & 0 & \ldots & 0 & j_1 & j_2 & j_3 & 0 & \ldots & 0\\
\vdots & \vdots & \ddots & \vdots & \vdots & \vdots & \vdots & \vdots & &\vdots\\
0 & 0 & \ldots & 0 & j_4 & j_5 & j_3 & 0 & \ldots & 0\\
0 & 0 & \ldots & 0 & j_4 & 0 & 0 & 0 & \ldots & 0\\
0 & 0 & \ldots & 0 & 0 & 0 & 0 & 0 & \ldots & 0 \\
\vdots & \vdots & \ddots & \vdots & \vdots & \vdots & \vdots & \vdots & &\vdots\\
0 & 0 & \ldots & 0 & 0 & 0 & 0 & 0 & \ldots & 0
\end{pmatrix},
\end{align*}
where $j_1=\frac{(\nu_{k+1}-1)\lambda_{k+1}}{\gamma_{k+1}-\lambda_{k+2}-\mu_{k+1}}$, $j_2=1+\frac{(1-\nu_{k+1})\lambda_{k+1}}{\gamma_{k}-\lambda_{k+1}}+\frac{\lambda_{k+2}-\gamma_{k+1}}{\gamma_{k+1}-\lambda_{k+2}-\mu_{k+1}}$, $j_3=-1+\frac{\lambda_{k+2}-\gamma_{k+1}}{\gamma_{k+1}-\lambda_{k+2}-\mu_{k+1}}$, $j_4=(1-\nu_{k+1})(\frac{\lambda_{k+1}}{\gamma_{k}-\lambda_{k+1}}-\frac{\lambda_{k+1}}{\gamma_{k+1}-\lambda_{k+2}-\mu_{k+1}})$, $j_5=\frac{(\gamma_{k}-\mu_{k+1}-\nu_{k+1})\lambda_{k+1}}{\gamma_{k}-\lambda_{k+1}}+\frac{\lambda_{k+2}-\gamma_{k+1}}{\gamma_{k+1}-\lambda_{k+2}-\mu_{k+1}}$.

\end{theorem}
\begin{proof} Note that, in the hypothesis, $\mathcal{L}_n=\mathcal{E}_n\mathcal{L}'_n=\mathcal{L}''_n\mathcal{E}_n$, where the matrices $\mathcal{L}'_n$ and $\mathcal{L}''_n$ can be computed by pre and post-multiplying $\mathcal{E}^{-1}_n$ in $\mathcal{L}_n$.  At first, we compute $\dot{\gamma}_{k}(t)=\dot{\lambda}_{k+1}(t)+\dot{c}_k(t)$, $\dot{\gamma}_{k+1}(t)=\dot{\lambda}_{k+2}(t)+\dot{c}_{k+1}(t)$ and $\dot{\gamma}_{k+2}(t)=\dot{\lambda}_{k+3}(t)+\dot{c}_{k+2}(t)$ using \eqref{PERT c_n} and \eqref{PERT lambda_n}. Now, $\mathcal{\dot{E}}_n$ is obtained as
\begin{align*}
\mathcal{\dot{E}}_n&= p(\mathcal{L}_n\mathcal{F}_n-\mathcal{F}_n\mathcal{L}_n)+q(\mathcal{L}_n\mathcal{G}_n-\mathcal{G}_n\mathcal{L}_n)+ p\mathcal{H}_n+q\mathcal{J}_n\\
&= p(\mathcal{E}_n\mathcal{L}'_n\mathcal{F}_n-\mathcal{F}_n\mathcal{L}''_n\mathcal{E}_n)+ q(\mathcal{E}_n\mathcal{L}'_n\mathcal{G}_n-\mathcal{G}_n\mathcal{L}''_n\mathcal{E}_n)+ p\mathcal{H}_n+q\mathcal{J}_n\\
&=\mathcal{E}_n(p\mathcal{L}'_n\mathcal{F}_n+q\mathcal{L}'_n\mathcal{G}_n)-(p\mathcal{F}_n\mathcal{L}''_n+q\mathcal{G}_n\mathcal{L}''_n)\mathcal{E}_n+ p\mathcal{H}_n+q\mathcal{J}_n
\end{align*}
which gives \eqref{MDE PERT} and the proof is complete.
\end{proof}
The study of matrix differential equations of the form \eqref{MDE PERT} is available in \cite{walter toda laurent 2002}. With $\mu_{k+1}=0$ and $\nu_{k+1}=1$, we get $\mathcal{L}_n=\mathcal{E}_n$, $\mathcal{J}_n=\mathcal{H}_n=\mathcal{O}_n$, $\mathcal{L}'_n=\mathcal{L}''_n=\mathcal{I}_n$, where $\mathcal{O}_n$ and $\mathcal{I}_n$ are the Zero matrix and Identity matrices, respectively. Then, \eqref{MDE PERT} becomes
\begin{align*}
	\mathcal{\dot{E}}_n=\mathcal{E}_n\mathcal{X}_n-\mathcal{X}_n\mathcal{E}_n,
\end{align*}
where $\mathcal{X}_n=p\mathcal{F}'_n+q\mathcal{G}'_n$, where $\mathcal{F}'_n$ and $\mathcal{G}'_n$ are the matrices $\mathcal{F}_n$ and $\mathcal{G}_n$ when $\mu_{k+1}=0$ and $\nu_{k+1}=1$. Thus, we have recovered the pair $[\mathcal{E}_n, \mathcal{X}_n]$ which is the Lax pair for extended relativistic Toda equations of finite order derived in \cite[Theorem 3]{Bracciali ranga toda 2019}.

\section{Stability of pertrbed system via numerical experiments}\label{Stability}
A system is said to be stable if perturbations to the system's initial state do not cause the system to diverge or become unpredictable over time. In other words, a stable system will return to its equilibrium state after being perturbed. The exploration of stability analysis holds significant research interest across various domains, including difference equations \cite{Atia Finkelshtein 2014}, differential equations \cite{Chen Srivastava 2023}, fluid mechanics \cite{Malik 2022}, etc. In this section, we make an attempt to establish connections between the methodologies developed in these fields and the framework of orthogonal polynomials.

Let
\begin{align*}
	c_n(t) = \sqrt{q} \quad {\rm and } \quad \lambda_n(t) = \dfrac{n-1}{2(t+d)}, \quad n \geq 2,
\end{align*}
be the recurrence coefficients of three term recurrence relation \eqref{R_n Toda} satisfied by $L$-orthogonal polynomials orthogonal with respect to measure $d\psi(x)= x^{-1/2}e^{-(t+d)(x+\frac{q}{x})}$, so that the Toda equations are found to be \cite{Bracciali ranga toda 2019}
\begin{align}\label{Unperturbed Toda equations}
	\dot{c}_n(t) = 0, \qquad \dot{\lambda}_n(t) = -\dfrac{n-1}{2(t+d)^2}, \quad d > 0, \quad n \geq 2.
\end{align}
These equations gets transformed according to the procedure outlined in \Cref{Perturbed theorem Toda} upon the recurrence coefficients' perturbation being introduced in the corresponding three-term recurrence relation \eqref{R_n Toda perturbed}. In this section, we delve into the analysis of the stability of the resulting system through graphical illustrations and numerical simulations, considering various choices of perturbations $\mu_k(t)$ and $\nu_k(t)$. These perturbations, being a function of $t$ where $t \geq 0$, can be either monotonically increasing (M.I.) or monotonically decreasing (M.D.) or oscillating (Osc.). As elucidated in \Cref{Main Remark 1}, the modification induced by perturbation occurs solely in $\dot{c}_k(t)$, $\dot{c}_{k+1}(t)$, $\dot{c}_{k+2}(t)$, $\dot{\lambda}_{k+1}(t)$ and $\dot{\lambda}_{k+2}(t)$ for $n=k$. For numerical experimentation, the perturbation is introduced at initial level $k=2$ and by virtue of \Cref{Perturbed theorem Toda}, we get
\begin{align*}
&\dot{c}_{2}(t)=\left(\dfrac{\sqrt{q}}{2(t+d)}-\dfrac{\nu_{3}(t)}{t+d}\right)+q\left(\dfrac{\nu_{3}(t)}{(\sqrt{q}-\mu_{3}(t))(t+d)}-\dfrac{1}{2\sqrt{q}(t+d)}\right), \\
&\dot{c}_{3}(t)=\dot{\mu}_{3}(t)+(\sqrt{q}-\mu_{3}(t))\left(\dfrac{\nu_{3}(t)}{t+d}-\dfrac{3}{2(t+d)}\right)+\sqrt{q}\left(\dfrac{3}{2(t+d)}-\dfrac{\nu_{3}(t)}{(t+d)} \right), \\
&\dot{c}_{4}(t)=-\dfrac{\sqrt{q}}{2(t+d)}+q\left(\dfrac{2}{\sqrt{q}(t+d)}-\dfrac{3}{(\sqrt{q}-\mu_{3}(t))(t+d)} \right), \\
& \dfrac{\dot{\lambda}_{3}(t)}{{\lambda}_{3}(t)}=-\dfrac{\dot{\nu}_{3}(t)}{{\nu}_{3}(t)}- \left(\dfrac{1}{t+d}+\sqrt{q}-\mu_{3}(t)+\dfrac{\nu_{3}(t)-1}{t+d}-\sqrt{q}\right)-q\left(\dfrac{1}{\sqrt{q}-\mu_{3}(t)}-\dfrac{1}{\sqrt{q}} \right),\\
&\dfrac{\dot{\lambda}_{4}(t)}{{\lambda}_{4}(t)}=-\left(\dfrac{2}{t+d}+\mu_{3}(t)-\dfrac{\nu_{3}(t)}{t+d}\right)-q\left(\dfrac{1}{\sqrt{q}}-\dfrac{1}{\sqrt{q}-\mu_{3}(t)} \right).
\end{align*}
For enhanced visual comprehension, the graphs exclusively depicting these equations are generated. Functions such as $e^t$, $e^{-t}$, $\sin t$, degree-one, and quadratic polynomials are employed as choices for $\mu_k(t)$ and $\nu_k(t)$. In the ensuing figures, "D$f(t)$" signifies $\dot{f}(t)$.

The unperturbed system corresponding to $\dot{c}_n(t)$ is represented by $x$-axis, while that corresponding to $\dot{\lambda}_n(t)$ is depicted in \Cref{6}(a), as obtained from \eqref{Unperturbed Toda equations}. All computations and graph plotting are executed using MATLAB$^{{\mbox{\tiny \textcircled{R}}}}$, utilizing an Intel Core i3-6006U CPU @ 2.00 GHz and 8 GB of RAM.

\subsection{Numerical Illustrations}
\begin{enumerate}
\item [1.] When $\mu_k(t)$ is monotonically decreasing and $\nu_k(t)$ be either monotonically increasing or decreasing or oscillating, the system for perturbed $\dot{c}_n(t)$ initially shows diversion from the original state but becomes stable as $t$ increases (see \Cref{1}(a), \Cref{1}(b), and \Cref{2}(a)). In other words, for any choice of $\nu_k(t)$, if $\mu_k(t)$ is monotonically decreasing, then the system attains a stable state eventually.

\begin{figure}[htb!]
	\centering
	\subfigure[]{\includegraphics[scale=0.5]{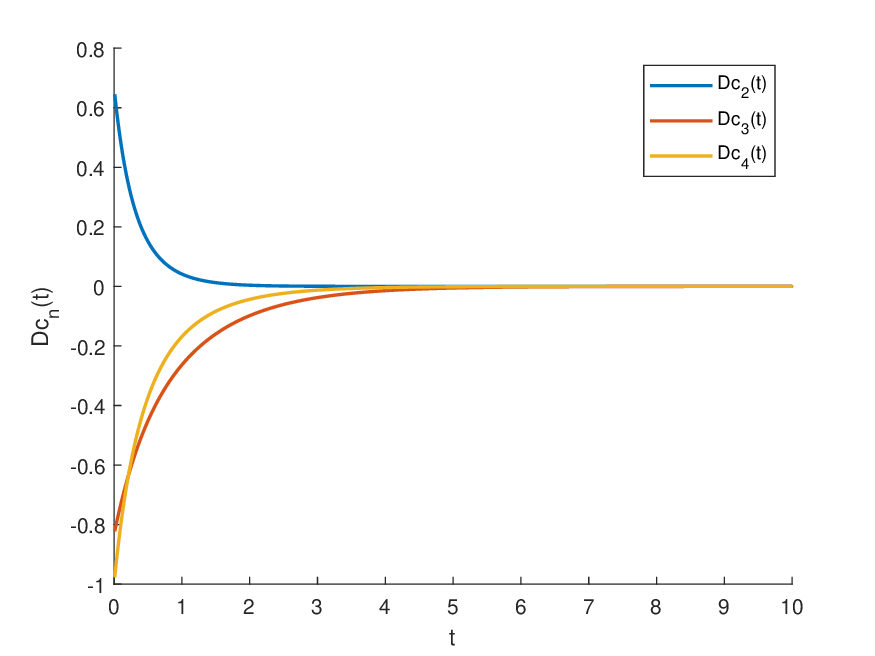}}
	\subfigure[]{\includegraphics[scale=0.5]{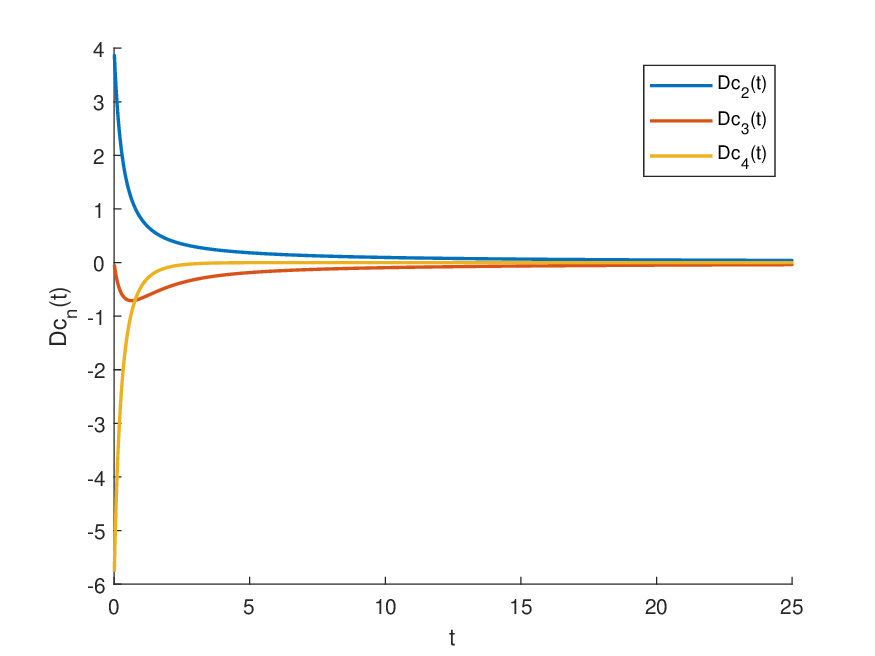}}
	\caption{(a) $\dot{c}_n(t)$ for $n=2,3,4$ when both $\mu_k(t)$ and $\nu_k(t)$ are M.D. (b) $\dot{c}_n(t)$ for $n=2,3,4$ when $\mu_k(t)$ is M.D. and $\nu_k(t)$ is M.I.}\label{1}
\end{figure}

\begin{figure}[htb!]
	\centering
	\subfigure[]{\includegraphics[scale=0.5]{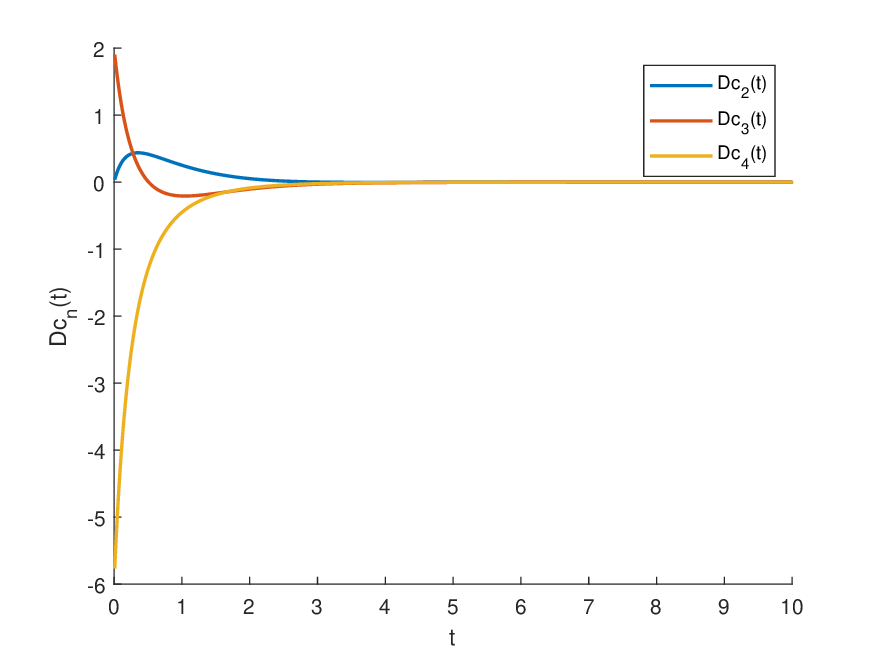}}
	\subfigure[]{\includegraphics[scale=0.5]{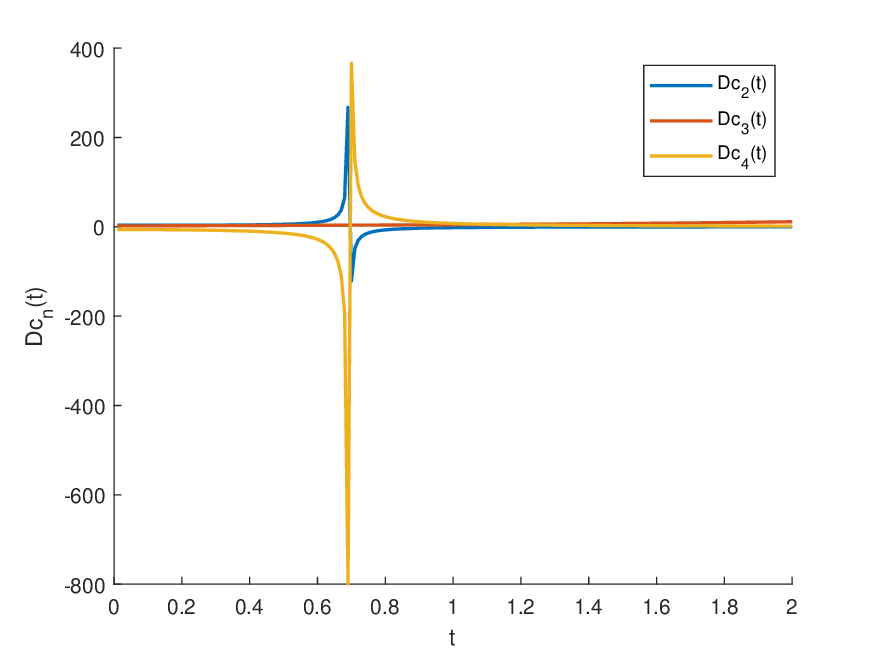}}
	\caption{(a) $\dot{c}_n(t)$ for $n=2,3,4$ when $\mu_k(t)$ is M.D. and $\nu_k(t)$ is Osc.  (b) $\dot{c}_n(t)$ for $n=2,3,4$ when $\mu_k(t)$ is M.I. and $\nu_k(t)$ is M.D.}\label{2}
\end{figure}

\begin{figure}[htb!]
	\centering
	\subfigure[]{\includegraphics[scale=0.5]{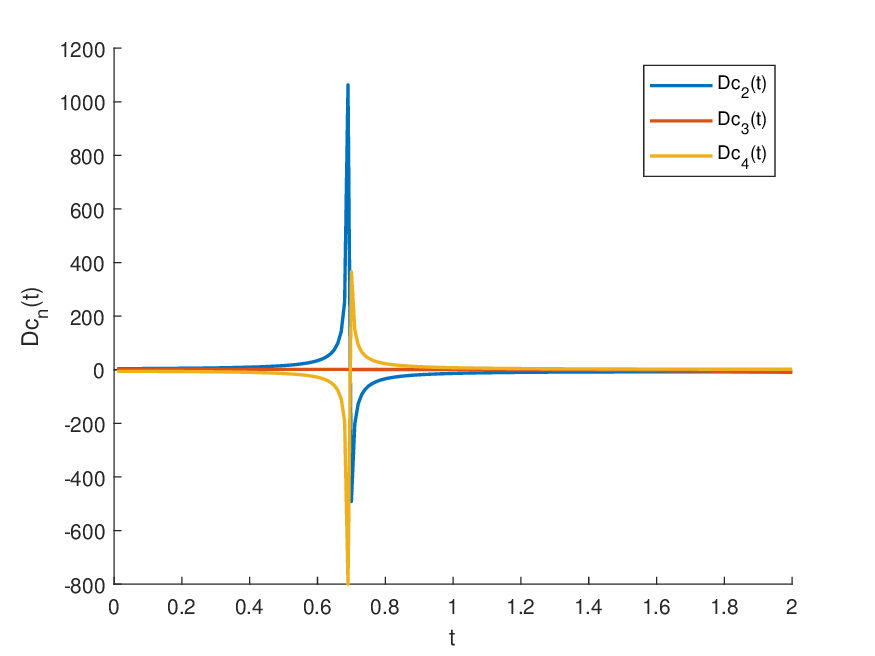}}
	\subfigure[]{\includegraphics[scale=0.5]{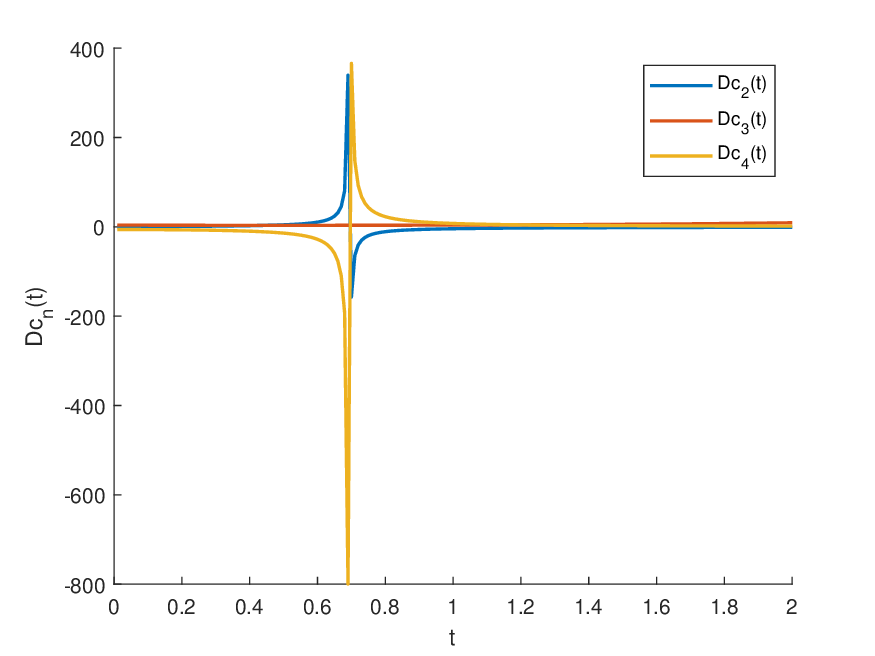}}
	\caption{(a) $\dot{c}_n(t)$ for $n=2,3,4$ when both $\mu_k(t)$ and $\nu_k(t)$ are M.I. ~(b) $\dot{c}_n(t)$ for $n=2,3,4$ when $\mu_k(t)$ is M.I. and $\nu_k(t)$ is Osc.}\label{3}
\end{figure}
\item [2.] \label{Note 2} When $\mu_k(t)$ is monotonically increasing, then for any choice of $\nu_k(t)$, the system for perturbed $\dot{c}_n(t)$ initially appears to be stable, but then experiences a sudden disturbance or perturbation that causes it to become temporarily unstable. However, the system eventually recovers from this disturbance and returns to a stable state (see \Cref{2}(b), \Cref{3}(a), and \Cref{3}(b)). The phenomenon is commonly known as {\it transient instability} or {\it transient instability followed by recovery to stability} \cite{Franklin Powell book,Nise book}.

\item [3.] When $\mu_k(t)$ is oscillating and $\nu_k(t)$ is either monotonically decreasing or oscillating, the system for perturbed $\dot{c}_n(t)$ experiences instability that leads to sustained oscillations without converging to a stable state whatsoever choice of $d$ and $t$ one makes. The oscillations persist indefinitely, and the system does not return to a stable condition (see \Cref{4}(a) and \Cref{4}(b)). The phenomenon is often referred to as {\it persistent oscillation} or {\it sustained oscillation} \cite{Franklin Powell book,Nise book}.

\begin{figure}[htb!]
	\centering
	\subfigure[]{\includegraphics[scale=0.5]{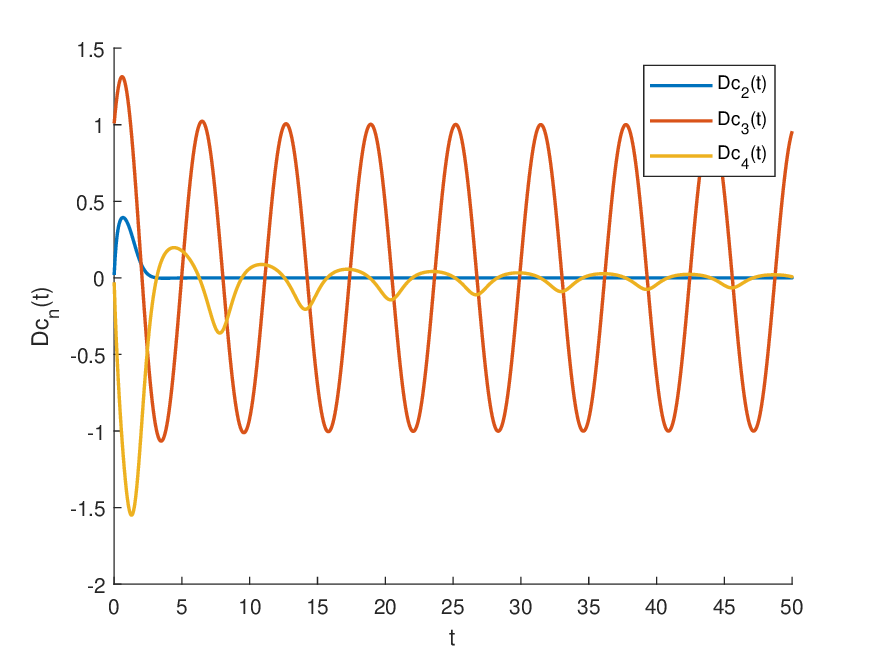}}
	\subfigure[]{\includegraphics[scale=0.5]{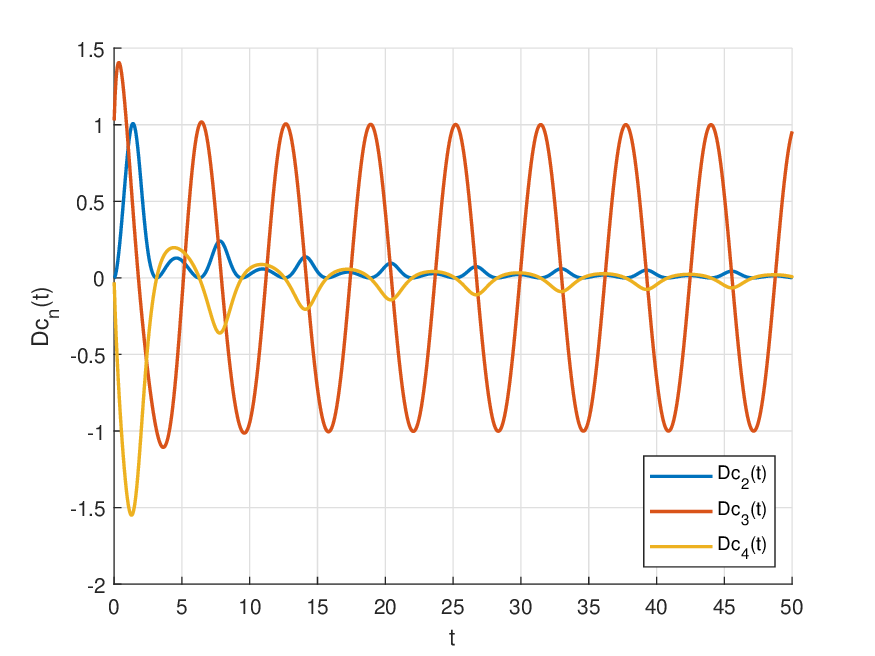}}
	\caption{(a) $\dot{c}_n(t)$ for $n=2,3,4$ when $\mu_k(t)$ is Osc. and $\nu_k(t)$ is M.D. ~(b) $\dot{c}_n(t)$ for $n=2,3,4$ when both $\mu_k(t)$ and $\nu_k(t)$ are Osc.}\label{4}
\end{figure}
\item [4.] For $\mu_k(t)$ oscillating and $\nu_k(t)$ monotonically increasing, the system for perturbed $\dot{c}_n(t)$ initially exhibits stability, but as $t$ increases, it becomes increasingly unstable, and the amplitude of oscillations grows continuously, eventually tending to infinity (see \Cref{5}). The phenomenon you're describing is commonly referred to as {\it instability with growing amplitude} or {\it exponential instability} \cite{Franklin Powell book,Nise book}. This is one of the case in which the stability can be controlled by varying the parameter $d$ (upto certain $t$).

\begin{figure}[htb!]
	\centering
	{\includegraphics[scale=0.5]{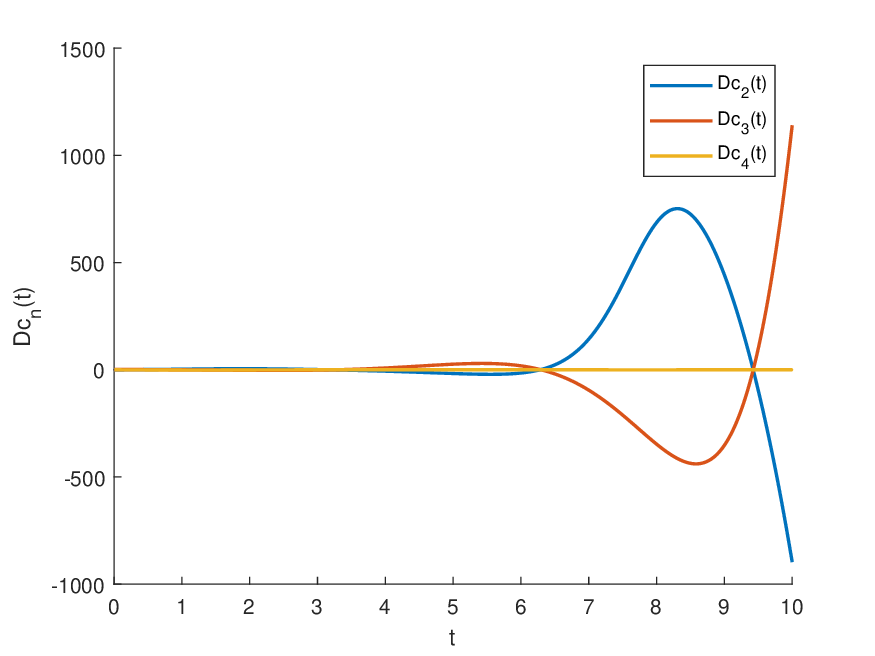}}
	\caption{ $\dot{c}_n(t)$ for $n=2,3,4$ when $\mu_k(t)$ is Osc. and $\nu_k(t)$ is M.I. }\label{5}
\end{figure}

\item [5.] When both $\mu_k(t)$ and $\nu_k(t)$ are monotonically decreasing, the system for perturbed $\dot{\lambda}_n(t)$ exhibits {\it transient response} where the system initially deviates from its original state but eventually stabilizes as $t$ increases (see \Cref{6}(b)).

\begin{figure}[htb!]
	\centering
	\subfigure[]{\includegraphics[scale=0.5]{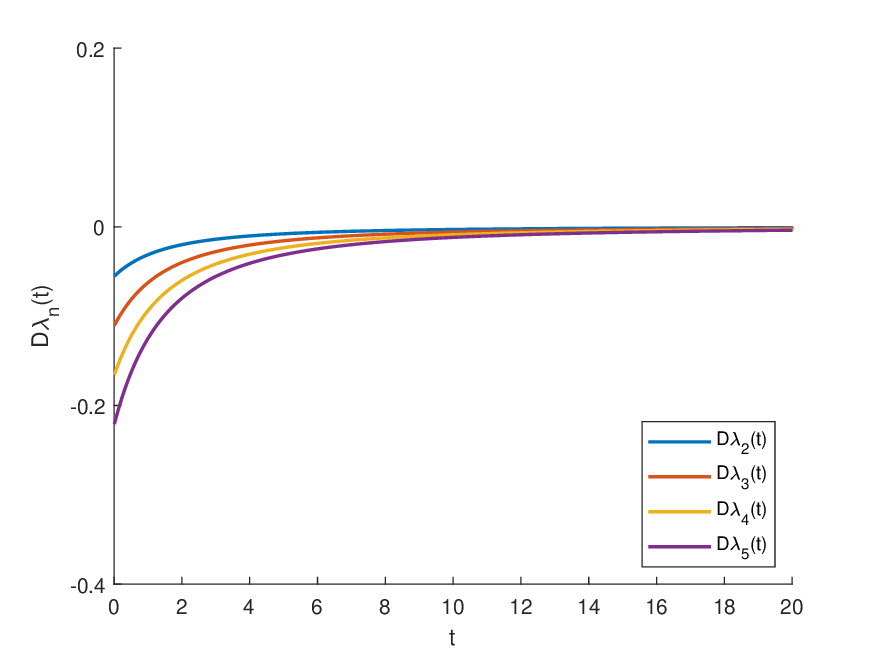}}
	\subfigure[]{\includegraphics[scale=0.5]{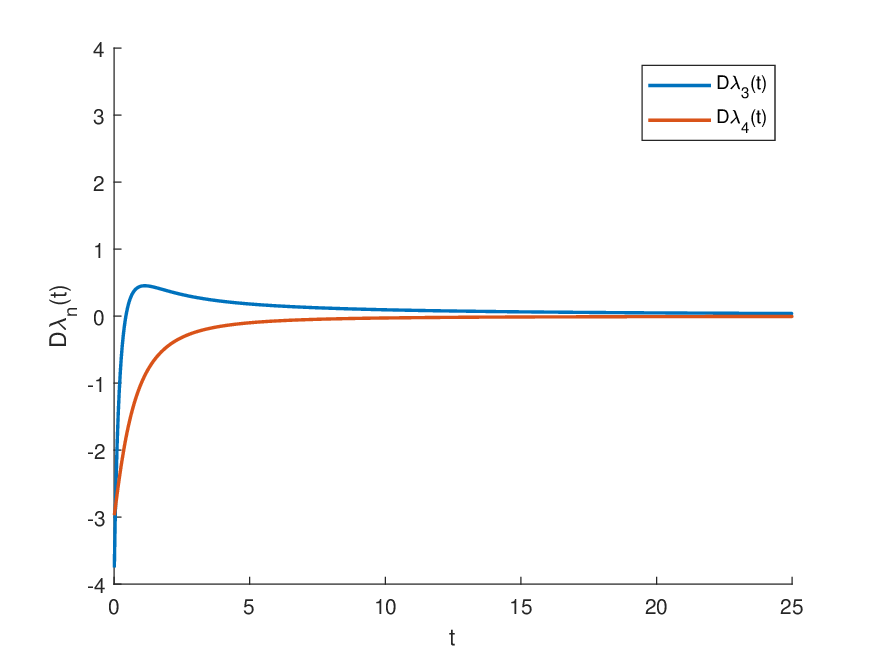}}
	\caption{(a) $\dot{\lambda}_n(t)$ for $n=2,3,4,5$ without any perturbation ~(b) $\dot{\lambda}_n(t)$ for $n=3,4$ when both $\mu_k(t)$ and $\nu_k(t)$ are M.D.}\label{6}
\end{figure}

\begin{figure}[htb!]
	\centering
	\subfigure[]{\includegraphics[scale=0.5]{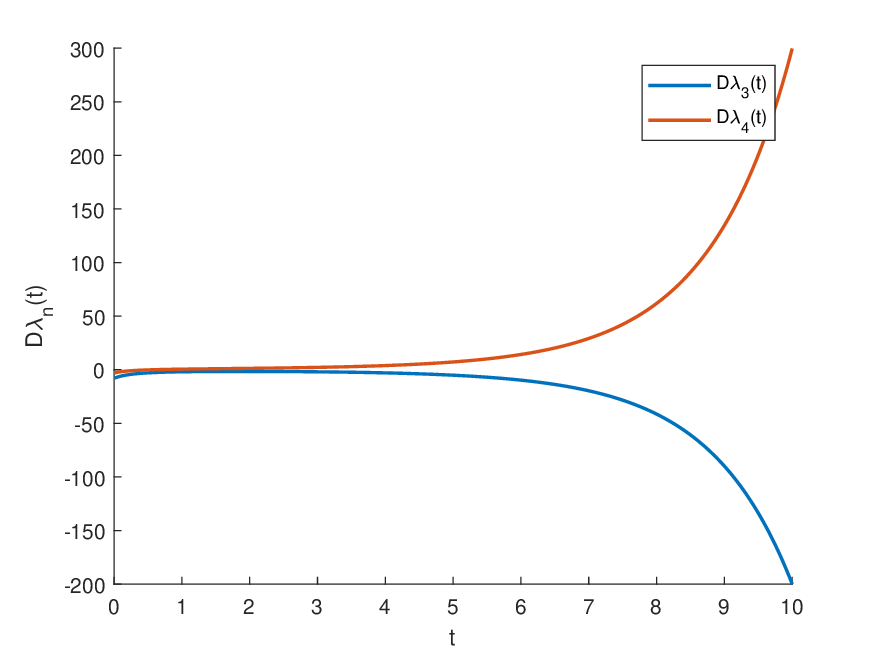}}
	\subfigure[]{\includegraphics[scale=0.5]{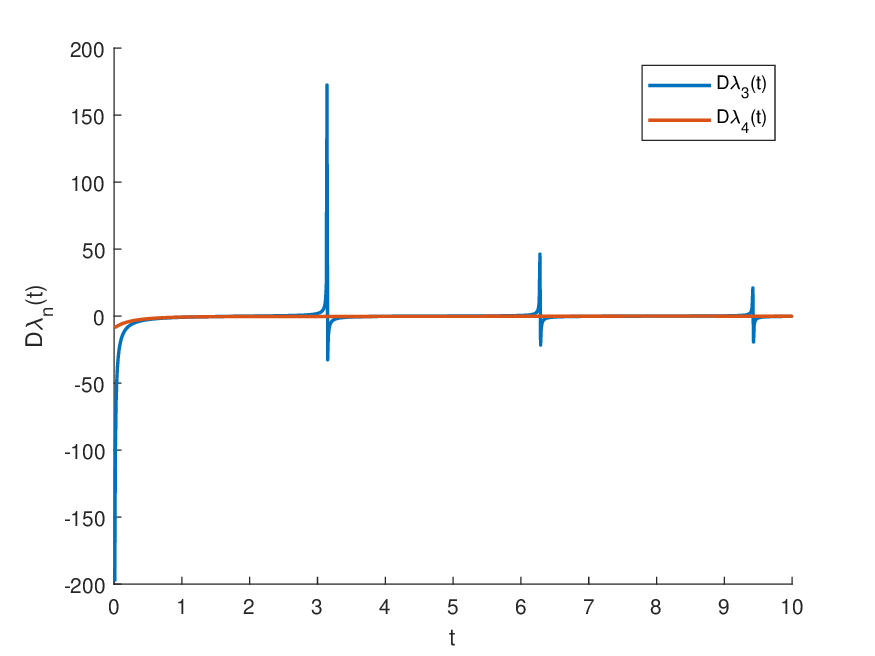}}
	\caption{(a) $\dot{\lambda}_n(t)$ for $n=3,4$ when $\mu_k(t)$ is M.D. and $\nu_k(t)$ is M.I. ~(b) $\dot{\lambda}_n(t)$ for $n=3,4$ when $\mu_k(t)$ is M.D. and $\nu_k(t)$ is Osc. }\label{7}
\end{figure}

\begin{figure}[htb!]
	\centering
	\subfigure[]{\includegraphics[scale=0.5]{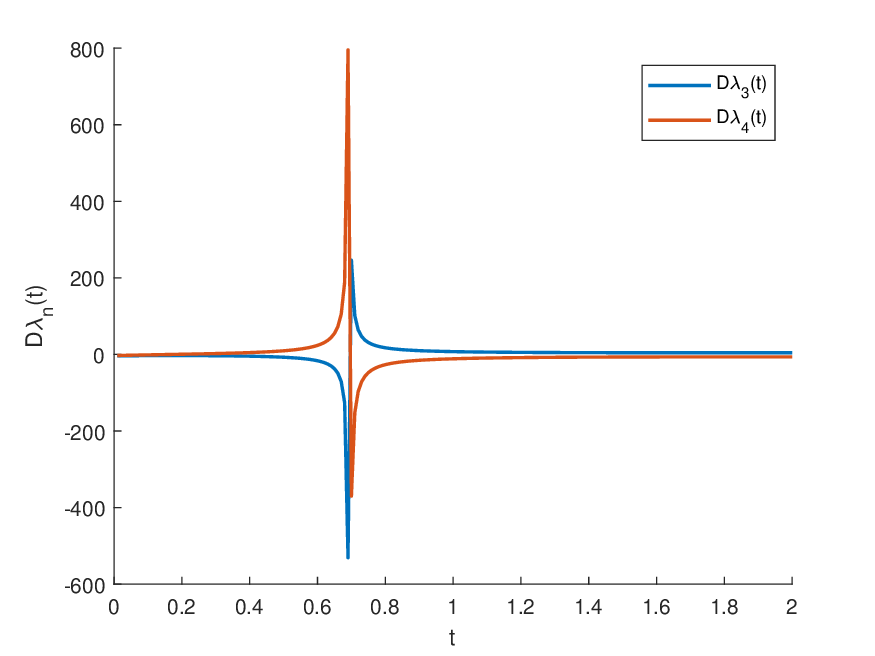}}
	\subfigure[]{\includegraphics[scale=0.5]{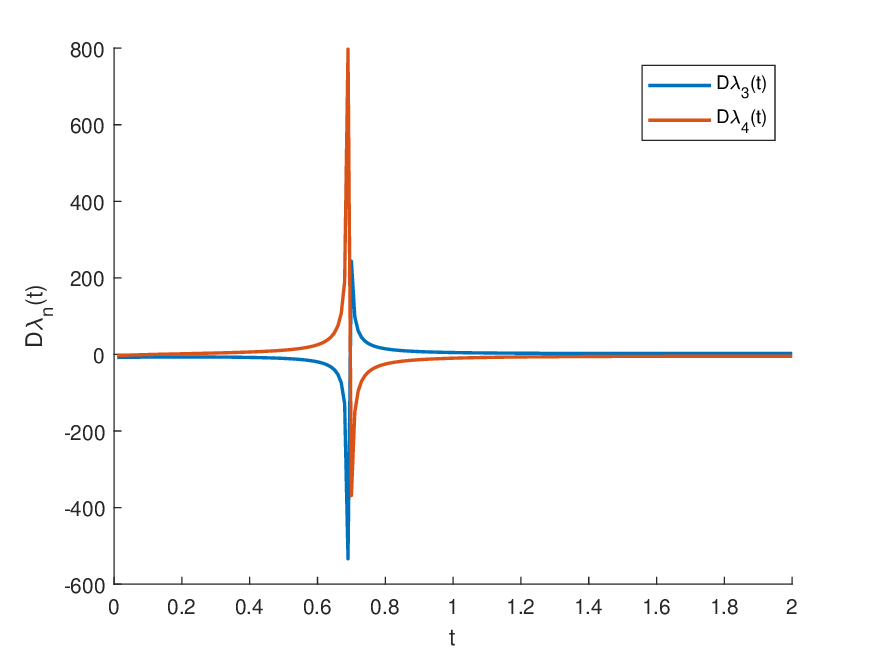}}
	\caption{(a) $\dot{\lambda}_n(t)$ for $n=3,4$ when $\mu_k(t)$ is M.I. and $\nu_k(t)$ is M.D. ~(b) $\dot{\lambda}_n(t)$ for $n=3,4$ when both $\mu_k(t)$ and $\nu_k(t)$ are M.I.}\label{8}
\end{figure}

\begin{figure}[htb!]
	\centering
	\subfigure[]{\includegraphics[scale=0.5]{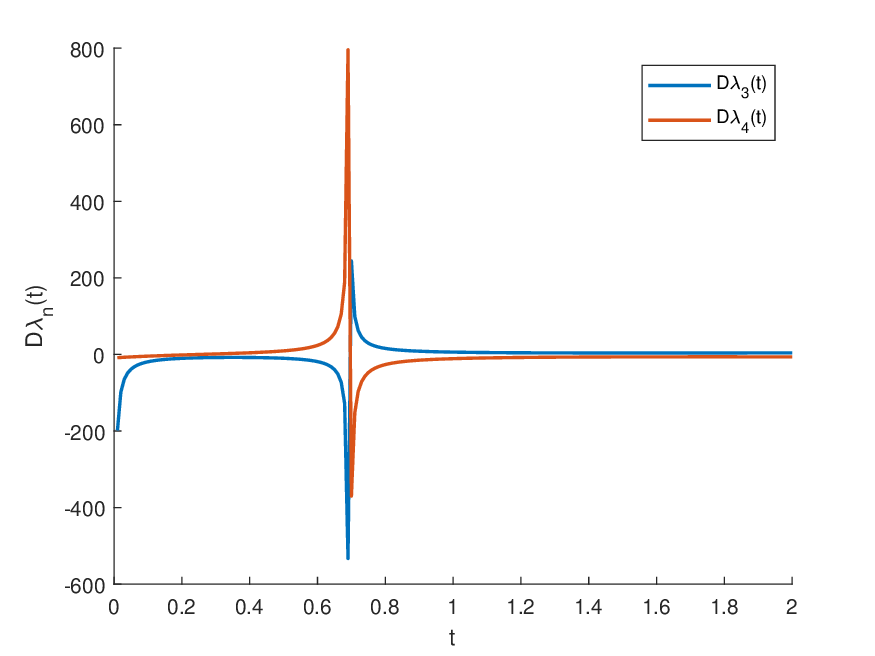}}
	\subfigure[]{\includegraphics[scale=0.5]{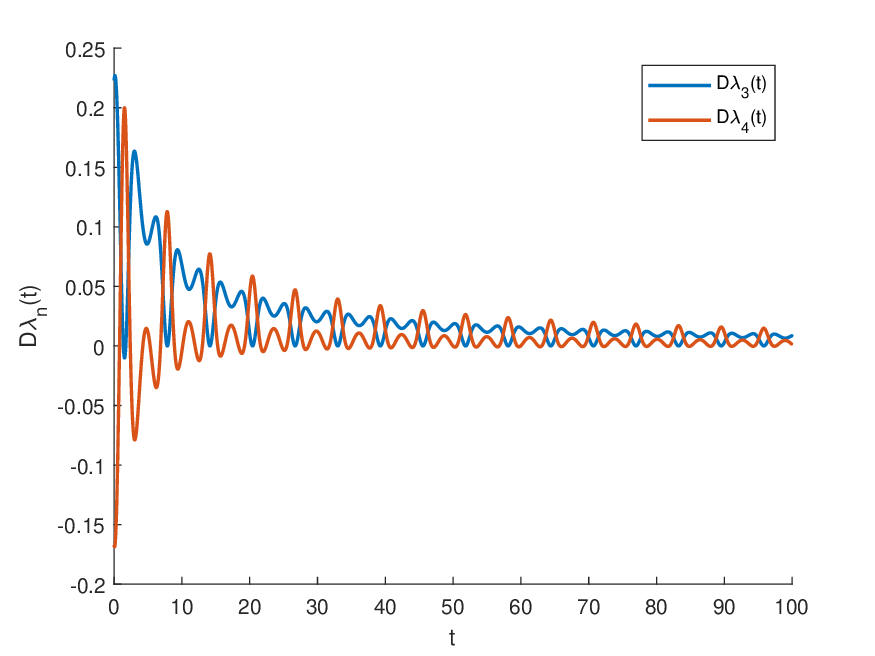}}
	\caption{(a) $\dot{\lambda}_n(t)$ for $n=3,4$ when $\mu_k(t)$ is M.I. and $\nu_k(t)$ is Osc. ~(b) $\dot{\lambda}_n(t)$ for $n=3,4$ when $\mu_k(t)$ is Osc. and $\nu_k(t)$ is M.D.}\label{9}
\end{figure}

\begin{figure}[htb!]
	\centering
	\subfigure[]{\includegraphics[scale=0.5]{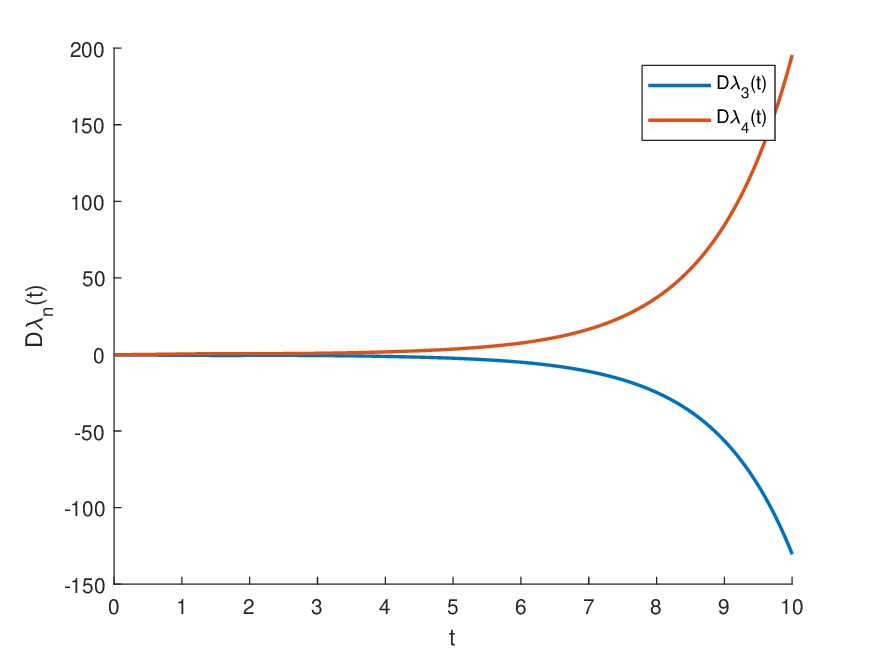}}
	\subfigure[]{\includegraphics[scale=0.5]{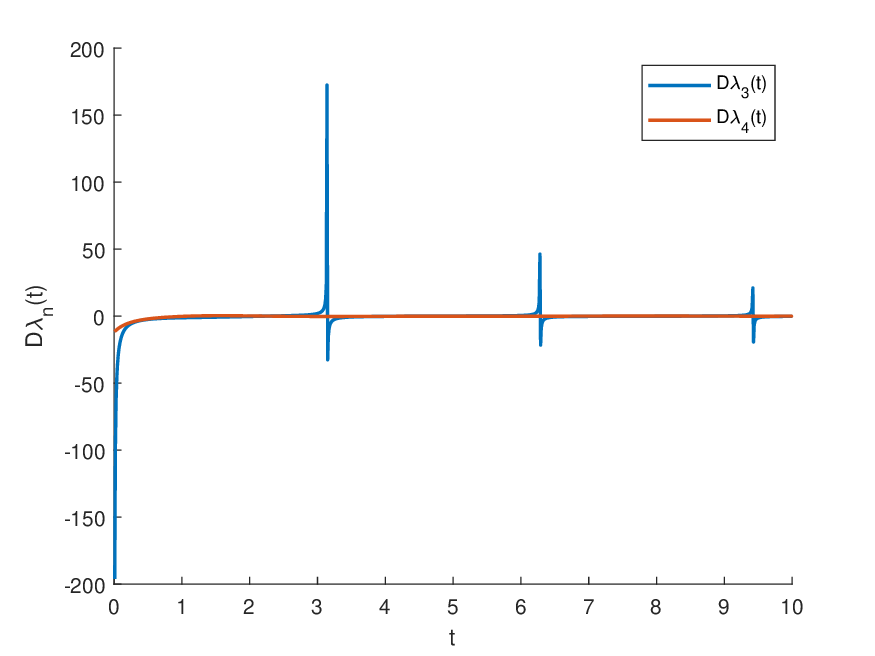}}
	\caption{(a) $\dot{\lambda}_n(t)$ for $n=3,4$ when $\mu_k(t)$ is Osc. and $\nu_k(t)$ is M.I. ~(b) $\dot{\lambda}_n(t)$ for $n=3,4$ when both $\mu_k(t)$ and $\nu_k(t)$ are Osc.}\label{10}
\end{figure}

\item [6.] When $\mu_k(t)$ is either monotonically decreasing or oscillaing but $\nu_k(t)$ monotonically increasing, the system for perturbed $\dot{\lambda}_n(t)$ remains stable for some time but then becomes unstable, but in this case, no oscillations are present (see \Cref{7}(a) and see \Cref{10}(a)), contrary to \Cref{5}. These are also the cases in which the stability upto certain $t$ can be controlled by varying the parameter $d$.
\item [7.] When $\mu_k(t)$ is either monotonically decreasing or oscillaing and $\nu_k(t)$ is also oscillating, the system's (perturbed $\dot{\lambda}_n(t)$) behavior alternates between stable periods and episodes of instability characterized by sharp spikes or disturbances (see \Cref{7}(b) and see \Cref{10}(b)). The intermittent nature of the instability suggests that there may be periodic factors or conditions influencing the system's dynamics, leading to the observed pattern of behavior. The phenomenon is often referred to as {\it intermittent instability} or {\it periodic instability} \cite{Franklin Powell book,Nise book}.
\item [8.] When $\mu_k(t)$ is monotonically increasing, then for any choice of $\nu_k(t)$, the system for perturbed $\dot{\lambda}_n(t)$ exhibits {\it transient instability} \cite{Franklin Powell book,Nise book} discussed in Note 2 (see \Cref{8}(a), \Cref{8}(b), and \Cref{9}(a)).
\item [9.] When $\mu_k(t)$ is oscillaing and $\nu_k(t)$ is monotonically decreasing, the (perturbed $\dot{\lambda}_n(t)$) system's response initially exhibits oscillatory behavior due to the disturbances, but over time, the amplitude of the oscillations decreases until the system settles into a stable state (see \Cref{9}(b)). The damping effect gradually reduces the magnitude of the disturbances until they no longer significantly affect the system's behavior, allowing it to achieve stability. The phenomenon is called {\it damped instability} \cite{Franklin Powell book,Nise book}.
\end{enumerate}
These numerical illustrations provide the advantage of understanding the type of stability (instability) that may occur at specific values, which in turn will be useful in handling the corresponding physical system effectively. On the other hand, a strong theoretical approach needs to be developed in obtaining the range of parameters involved in the stability behaviour of the physical system given by \Cref{Perturbed theorem Toda}.

\end{document}